\documentclass[12pt]{amsart}

\usepackage{amssymb,latexsym}

\usepackage{enumerate}

\makeatletter

\@namedef{subjclassname@2010}{

  \textup{2010} Mathematics Subject Classification}

\makeatother
\newtheorem{thm}{Theorem}[section]

\newtheorem{cor}[thm]{Corollary}
\newtheorem{lem}[thm]{Lemma}
\newtheorem{pro}[thm]{Proposition}
\theoremstyle{definition}

\numberwithin{equation}{section}

\newcommand{\X}{\mathbb{X}}

\newcommand{\ex}{\mathbb{E}}
\newcommand{\re}{\textup{Re}}
\newcommand{\im}{\textup{Im}}
\newcommand{\F}{\mathcal{F}}
\newcommand{\mc}{\mathcal}
\newcommand{\sums}{\sideset{}{^\flat}\sum}

\begin{document}

\baselineskip=17pt

\title[Imaginary quadratic fields with prime discriminant]{The number of imaginary quadratic fields with prime discriminant and class number up to $H$}

\author{Youness Lamzouri}

\address{Department of Mathematics and Statistics,
York University,
4700 Keele Street,
Toronto, ON,
M3J1P3
Canada}

\email{lamzouri@mathstat.yorku.ca}

\date{}

\begin{abstract}  
 In this paper, we obtain an asymptotic formula for the number of imaginary quadratic fields with prime discriminant and class number up to $H$, as $H\to \infty$. Previously, such an asymptotic was only known under the assumption of the Generalized Riemann Hypothesis, by the recent work of Holmin, Jones,  Kurlberg, McLeman and Petersen.

\end{abstract}


\thanks{The author is partially supported by a Discovery Grant from the Natural Sciences and Engineering Research Council of Canada.}

\maketitle

\section{Introduction}

The celebrated Gauss class number problem, posed by Gauss in his Disquisitiones Arithmeticae of 1801, asks for the determination of all imaginary quadratic fields with a given class number. This was solved for $h=1$ by Baker, Heegner and Stark in the 50's and 60's, by Baker and Stark for $h=2$, and by Oesterl\'e for $h=3$.  We now have a complete list of all imaginary quadratic fields with class number $h$ for all $h\leq 100$ thanks to the work of Watkins \cite{Wa}.

In \cite{So} Soundararajan proved that there are asymptotically $\frac{3\zeta(2)}{\zeta(3)} H^2$ imaginary quadratic fields with class number up to $H$, as $H\to \infty$, where $\zeta(s)$ is the Riemann zeta function. He also studied the quantity $\F(h)$ defined as the number of imaginary quadratic fields with class number $h$. In particular, a more precise form of his asymptotic formula asserts that
\begin{equation}\label{SoundAsymp} 
\sum_{h\leq H} \F(h)= \frac{3\zeta(2)}{\zeta(3)} H^2 + O_{\varepsilon}\left(\frac{H^2}{(\log H)^{1/2-\varepsilon}}\right).
\end{equation}
The error term was recently improved to $H^2(\log H)^{-1-\varepsilon}$ by the author in \cite{La2}. 

Furthermore, Soundararajan \cite{So}  conjectured that for large $h$ we have
$$\frac{h}{\log h}\ll \F(h)\ll h\log h,$$
where the variation in size depends on the largest power of $2$ that divides $h$.
In particular, when $h$ is odd, he conjectured that 
$$ \F(h)\asymp \frac{h}{\log h},$$
and noted that the precise constant would depend on the arithmetic properties of $h$. 
In their recent investigation of class groups of imaginary quadratic fields, Holmin, Jones,  Kurlberg, McLeman and Petersen \cite{HJKMP} refined this conjecture to an asymptotic estimate. More precisely, they conjectured that as $h\to \infty$ through odd values, we have
$$\F(h) \sim  \mathcal{C}\cdot c(h)\cdot \frac{h}{\log h}   $$
where 
$$ \mathcal{C}=15 \prod_{p>2}\prod_{i=2}^{\infty}\left(1-\frac{1}{p^i}\right) 
\approx 11.317, \ \  \text{ and } \ \ c(h)=\prod_{p^n \| h}\prod_{i=1}^n \left(1-\frac{1}{p^i}\right)^{-1}.$$
To obtain this conjecture, they used the Cohen-Lenstra heuristics, together with a similar asymptotic formula to \eqref{SoundAsymp} averaged over odd values of $h$,
which they proved assuming the Generalized Riemann Hypothesis GRH. More precisely, Theorem 1.5 of \cite{HJKMP} states that conditionally on GRH, we have
\begin{equation}\label{SumOdd}\sum_{\substack{h\leq H\\ h  \text{ odd}}}\F(h)=\frac{15}{4}\cdot \frac{H^2}{\log H} +
O_{\varepsilon}\left(\frac{H^2}{(\log H)^{3/2-\varepsilon}}\right).
\end{equation}
By genus theory, if $d<-8$ is a fundamental discriminant, then the class number of the imaginary quadratic field $\mathbb{Q}(\sqrt{d})$ is odd if and only if $-d$ is prime. Furthermore,  note that the only composite fundamental discriminants $d$ with $-8\leq d<0$ are $-4$ and $-8$ which have class number $1$. Therefore, the number of imaginary quadratic fields with prime discriminant and class number up to $H$ equals
$\sum_{\substack{h\leq H\\ h  \text{ odd}}}\F(h) -2,$
and hence \eqref{SumOdd} might be viewed as a conditional asymptotic formula for this quantity as $H\to \infty$. The goal of the present paper is to establish  \eqref{SumOdd} unconditionally. 
\begin{thm}\label{AverageClass} Let $H$ be large. Then
\begin{equation}\label{MainAsymp}
\sum_{\substack{h\leq H\\ h  \text{ odd}}}\F(h)=\frac{15}{4}\cdot \frac{H^2}{\log H} +
O\left(\frac{H^2 (\log\log H)^3}{(\log H)^{3/2}}\right).
\end{equation}
\end{thm}
\noindent Note that assuming GRH, the error term above can be improved to $H^2 (\log\log H)^3/(\log H)^{2}$ (see \cite{La2}).

 Let $h(d)$ denote the class number of the quadratic field $\mathbb{Q}(\sqrt{d})$. Dirichlet's class number formula for imaginary quadratic fields asserts that
\begin{equation}\label{Class}
h(d) =\frac{\sqrt{|d|}}{\pi} L(1,\chi_d),
\end{equation}
for all fundamental discriminants $d<-4$, where $\chi_d=\left(\frac{d}{\cdot}\right)$ is the Kronecker symbol, and $L(s,\chi_d)$ is the Dirichlet $L$-function attached to $\chi_d$. Hence, the distribution of $h(d)$ is ultimately connected to that of $L(1, \chi_d)$.

To obtain the conditional estimate \eqref{SumOdd}, the authors of \cite{HJKMP} followed the approach in \cite{So}, which relies on computing the complex moments of $L(1,\chi_d)$. Using the ideas of Granville and Soundararajan \cite{GrSo}, Holmin, Jones,  Kurlberg, McLeman and Petersen \cite{HJKMP} computed the complex moments of $L(1,\chi_d)$ as $d$ varies in $$\mc{P}(x):=\{ d=-p: p\leq x \text{ is prime, and } p\equiv 3 \bmod 4\},
$$
conditionally on GRH. To describe their result, we need some notation. Let $\X(p)$ be a sequence of independent identically distributed random variables such that  $\X(p)=\pm 1$ with equal probabilities $1/2$. Consider the random product
$$ L(1,\X) :=\prod_{p} \left(1-\frac{\X(p)}{p}\right)^{-1},$$
which converges almost surely by Kolmogorov's three series theorem. Then, Theorem 3.3 of \cite{HJKMP} asserts that assuming GRH, for all complex numbers $|z|\leq \log x/(50(\log\log x)^2)$ we have
$$
\sum_{d\in \mathcal{P}(x)} L(1, \chi_d)^z=   |\mathcal{P}(x)|\cdot \ex\left(L(1,\X)^{z}\right)+ O_{\varepsilon}\left(x^{1/2+\varepsilon}\right).
$$
Using a different approach, we obtain an unconditional version of this asymptotic formula, though in a smaller range. One of the difficulties in obtaining such a result unconditionally arises from the possible existence of Landau-Siegel zeros. In this case, we isolate an extra factor in the asymptotic which comes from a single exceptional modulus defined as follows. By Chapter 20 of \cite{Da}, there is at most one square-free integer $q_1$ such that $|q_1|\leq \exp(\sqrt{\log x})$ and $L(s, \chi_{q_1})$ has a zero in the region 
\begin{equation}\label{ExRegion}
\re(s)>1-\frac{c}{\sqrt{\log x}},
\end{equation}
 for some positive constant $c$. Moreover, this zero $\beta$ (if it exists) is unique, real and simple.
\begin{thm}\label{ComplexMoments} Let $x$ be large. Then for all complex numbers $z$ such that $\re(z)\geq -1$ and $|z|\leq \sqrt{\log x}/(\log\log x)^2$ we have
\begin{align*}
\sum_{d\in \mathcal{P}(x)} L(1, \chi_d)^z= \ & \frac{\textup{Li}(x)}{2} \cdot \ex\left(L(1,\X)^{z}\right)-\textup{sgn}(q_1) \displaystyle{\frac{\textup{Li}(x^{\beta})}{2}} \cdot \ex\Big(\X(|q_1|)\cdot L(1,\X)^{z}\Big)\\
& + O\left(x\exp\left(-\frac{\sqrt{\log x}}{5\log\log x}\right)\right),
\end{align*}
where $\textup{sgn}(q_1)$ is the sign of $q_1$.
\end{thm}

If such an exceptional discriminant $q_1$ exists, then we must have $|q_1|\geq (\log x)^{1-o(1)}$ (see Chapter 20 of \cite{Da}). This allows us to prove that the contribution of the secondary term in Theorem \ref{ComplexMoments} to the asymptotic estimate of $\sum_{\substack{h\leq H\\ h  \text{ odd}}}\F(h)$ in Theorem \ref{AverageClass} is negligible. 
We also note that using Theorem \ref{ComplexMoments} together with Soundararajan's method \cite{So} (as in the proof of \eqref{SumOdd} in \cite{HJKMP}) produces only a saving of $(\log H)^{1/4-\varepsilon}$ in the error term of Theorem \ref{AverageClass}, since the range of validity of the asymptotic in Theorem \ref{ComplexMoments} is reduced to $|z|\leq (\log x)^{1/2-o(1)}$. Instead, we use the approach of \cite{La2} in order to obtain the improved saving of $(\log H)^{1/2-\varepsilon}$ in Theorem \ref{AverageClass}, which matches that of the conditional estimate \eqref{SumOdd}.

\section{Preliminary results}
Let $\chi\pmod q$ be a Dirichlet character, and $z\in \mathbb{C}$. For all complex numbers $s$ with $\re(s)>1$ we have 
$$ L(s, \chi)^z= \sum_{n=1}^{\infty} \frac{d_z(n)}{n^s}\chi(n),$$
where $d_z(n)$ is the $z$-th divisor function, defined as the multiplicative function such that $d_z(p^a)=\Gamma(z+a)/(\Gamma(z)a!)$ for all primes $p$ and positive integers $a$. We shall need the following bounds for these divisor functions and their sums. First, note that
\begin{equation}\label{Divisor1}
|d_z(n)|\leq d_{|z|}(n)\leq d_k(n)
\end{equation}
for any integer $k\geq |z|$, and $d_k(mn)\leq d_k(m)d_k(n)$ for any positive integers $k,m,n$.  Furthermore,
for $k\in {\Bbb N}$, and $y>3$ we have 
 $$
 d_k(n)e^{-n/y}\leq
e^{k/y}\sum_{a_1...a_k=n}e^{-(a_1+...+a_k)/y},
$$ and hence
\begin{equation}\label{Divisor2}
 \sum_{n=1}^{\infty}\frac{d_k(n)}{n}e^{-n/y}\leq \left(e^{1/y}
\sum_{a=1}^{\infty}\frac{e^{-a/y}}{a}\right)^k\leq (\log
2y)^k.
\end{equation}
We will also need the following bound,  which follows from Lemma 3.3 of \cite{La1} 
\begin{equation}\label{Divisor3}
 \sum_{n=1}^{\infty}\frac{d_k(n)^2}{n^{2-\delta}} \leq \exp\left((2+o(1)) k\log\log k \right),
\end{equation}
where $k$ is a large positive integer, and $\delta$ is any positive real number such that $\delta\leq 1/(2\log k)$.

Let $\X(p)$ be a sequence of independent identically distributed random variables such that  $\X(p)=\pm 1$ with equal probabilities $1/2$. 
We extend the $\X(p)$'s multiplicatively to all positive integers by setting $\X(1)=1$ and 
$ \X(n):= \X(p_1)^{a_1}\cdots \X(p_k)^{a_k} $ if $n= p_1^{a_1}\cdots p_k^{a_k}.$ Since $\ex(\X(p^a))=\ex(\X(p)^a)=1$ if $a$ is even, and equals $0$ if $a$ is odd, then for any positive integer $\ell$ we have $\ex(\X(\ell))=1$ if $\ell$ is a square, and equals $0$ otherwise. For $\sigma>1/2$ we define the random Euler product 
$$ L(\sigma, \X):=\prod_{p} \left(1-\frac{\X(p)}{p^{\sigma}}\right)^{-1}$$
which converges almost surely by Kolmogorov's three series theorem. Let $1/2<\sigma\leq 1$, $z\in \mathbb{C}$, and $q\geq 1$ be a square-free number. Then, we have almost surely $$
L(\sigma, \X)^{z}= \sum_{n=1}^{\infty} \frac{d_z(n)}{n^{\sigma}}\X(n),
$$
and hence
$$ \ex\left(\X(q)\cdot L(\sigma, \X)^{z}\right)= \ex\left(\X(q)\sum_{n=1}^{\infty} \frac{d_z(n)}{n^{\sigma}}\X(n)\right)= \sum_{n=1}^{\infty} \frac{d_z(n)}{n^{\sigma}}\ex\left(\X(qn)\right).$$
Moreover, since $q$ is square-free, then $qn$ is a square if and only if $n=q m^2$ for some integer $m\geq 1$. Thus, we obtain
\begin{equation}\label{RandomExpectation1}
\ex\left(\X(q)\cdot L(\sigma, \X)^{z}\right)= \sum_{m=1}^{\infty} \frac{d_z(qm^2)}{(qm^2)^{\sigma}}.
\end{equation}
In particular, since $\X(q)\leq 1$ and $L(\sigma, \X)>0$ almost surely, then for any real number $k$ we have
\begin{equation}\label{BoundRandom}
\sum_{m=1}^{\infty} \frac{d_k(qm^2)}{(qm^2)^{\sigma}}= \ex\left(\X(q)\cdot L(\sigma, \X)^{k}\right)\leq \ex\left(L(\sigma, \X)^{k}\right)=  \sum_{m=1}^{\infty} \frac{d_k(m^2)}{m^{2\sigma}}.
\end{equation}
Furthermore, observe that
\begin{equation}\label{RandomExpectation2}
\begin{aligned}
&\ex\left(\X(q)\cdot L(1,\X)^z\right)
=\prod_{p \mid q} \ex\left(\X(p)\left(1-\frac{\X(p)}{p}\right)^{-z}\right)\prod_{p\nmid q} \ex\left(\left(1-\frac{\X(p)}{p}\right)^{-z}\right)\\
&=\prod_{p \mid q} \left(\frac 12 \left(1-\frac{1}{p}\right)^{-z}-\frac 12 \left(1+\frac{1}{p}\right)^{-z}\right)\prod_{p\nmid q}\left(\frac 12 \left(1-\frac{1}{p}\right)^{-z}+\frac 12 \left(1+\frac{1}{p}\right)^{-z}\right)\\
&= \ex\left(L(1,\X)^z\right)\prod_{p \mid q}\left(\frac{(p-1)^{-z}-(p+1)^{-z}}{(p-1)^{-z}+(p+1)^{-z}}\right)\\
&= \ex\left(L(1,\X)^z\right)\prod_{p \mid q}\left(\frac{(p+1)^{z}-(p-1)^{z}}{(p+1)^{z}+(p-1)^{z}}\right).
\end{aligned}
\end{equation}

In \cite{La0}, the author studied the distribution of a large class of random models, which includes $L(1,\X)$. In particular, it follows from Theorem 1 of \cite{La0} that there is an explicit constant $A$, such that for large $\tau$ we have
\begin{equation}\label{LargeDeviation}
\mathbb{P}(L(1, \X)\geq e^{\gamma} \tau)= \exp\left(-\frac{e^{\tau-A}}{\tau}\left(1+O\left(\frac{1}{\sqrt{\tau}}\right)\right)\right),
\end{equation}
and 
\begin{equation}\label{LargeDeviation2}
\mathbb{P}\left(L(1, \X)\leq \zeta(2)(e^{\gamma} \tau)^{-1}\right)= \exp\left(-\frac{e^{\tau-A}}{\tau}\left(1+O\left(\frac{1}{\sqrt{\tau}}\right)\right)\right),
\end{equation}
where $\gamma$ is the Euler-Mascheroni constant. These large deviation estimates will be used in the proof of Theorem \ref{AverageClass}.

In order to compute the complex moments of $L(1, \chi_d)$ over $d\in \mc{P}(x)$ and prove Theorem \ref{ComplexMoments}, we need to estimate the character sum $\sum_{d\in \mc{P}(x)} \chi_{d}(n)$. By the law of quadratic reciprocity, this amounts to estimating the character sum over primes $\sum_{p\leq x} \chi_n (p)$. It follows from  Chapter 20 of \cite{Da}, that for all square-free integers $|n|\leq \exp(\sqrt{\log x})$ with at most one exception $q_1$, we have
\begin{equation}\label{NoExceptional}
\sum_{p\leq x} \chi_n(p) \ll x\exp\left(-c\sqrt{\log x}\right),
\end{equation}
for some positive constant $c$. Furthermore, for this exceptional $q_1$ (if it exists), the associated $L$-function $L(s, \chi_{q_1})$ has a unique real simple zero $\beta$, such that 
$\beta>1-c/\sqrt{\log x},$
and moreover we have
\begin{equation}\label{Exceptional}
\sum_{p\leq x} \chi_{q_1}(p)= -\textup{Li}\left(x^\beta\right) + x\exp\left(-c\sqrt{\log x}\right).
\end{equation}
\begin{lem}\label{CharSum}
Let $x$ be large and $n\leq \exp\left(\sqrt{\log x}\right)$ be a positive integer. Then, we have
$$
\sum_{d\in \mc{P}(x)} \chi_{d}(n)=\begin{cases} \displaystyle{\frac{\textup{Li}(x)}{2}} +O\left(x\exp\left(-c\sqrt{\log x}\right)\right), & \textup{ if } n= m^2,\\ \\
-\textup{sgn}(q_1) \displaystyle{\frac{\textup{Li}(x^{\beta})}{2}} +O\left(x\exp\left(-c\sqrt{\log x}\right)\right), &\textup{ if } n=  |q_1| \cdot m^2,\\ \\
O\left(x\exp\left(-c\sqrt{\log x}\right)\right), & \textup{ otherwise}.\end{cases}
$$
\end{lem}

\begin{proof}
First, we have 
$$\sum_{d\in \mc{P}(x)} \chi_{d}(n)=\sum_{\substack{p\leq x\\ p\equiv 3 \bmod 4}} \left(\frac{-p}{n}\right).$$
Write $n=n_1 m^2 $ where $n_1$ is square-free. Then, it follows from the law of quadratic reciprocity that for any prime $p\equiv 3\bmod 4$ such that $p\nmid n$, we have
$$\left(\frac{-p}{n}\right)= \left(\frac{n}{p}\right)=\left(\frac{n_1}{p}\right).$$ Thus, we get
$$\sum_{d\in \mc{P}(x)} \chi_{d}(n)= \sum_{\substack{p\leq x\\ p\equiv 3 \bmod 4}} \left(\frac{n_1}{p}\right) +O\left(\omega(n)\right)= 
\frac{1}{2} \sum_{p\leq x} \left(\frac{n_1}{p}\right)- \frac{1}{2}\sum_{p\leq x} \left(\frac{-n_1}{p}\right) + O\left(\omega(n)\right).$$
The first estimate, which corresponds to the case $n_1=1$, follows simply from the prime number theorem in arithmetic progressions. 
Now, if $n_1=|q_1|$, then we get
$$\sum_{p\leq x} \left(\frac{n_1}{p}\right)- \sum_{p\leq x} \left(\frac{-n_1}{p}\right)= -\textup{sgn}(q_1)\cdot\textup{Li}(x^{\beta})+O\left(x\exp\left(-c\sqrt{\log x}\right)\right),
$$
by \eqref{NoExceptional} and \eqref{Exceptional}. The final estimate follows from \eqref{NoExceptional}.
\end{proof}

Note that Lemma \ref{CharSum} is valid only in the small range $n\leq \exp(\sqrt{\log x})$. Hence, in order to use this result in the proof of Theorem \ref{ComplexMoments}, we need to find an approximation of the form
$$L(1,\chi_d)^z\approx\sum_{n\leq y} \frac{d_z(n)\chi_d(n)}{n}$$
where $y\leq \exp(\sqrt{\log x})$. The following result, which is a slightly different version of Proposition 3.3  of \cite{DaLa}, shows that we can find a good approximation to $L(1,\chi_d)^z$ if
 $L(s,\chi_d)$ has no zeros in a certain small rectangle near $1$. The proof is similar to that of Proposition 3.3  of \cite{DaLa}, but we shall include it for the sake of completeness.  Here and throughout we let $\log_j$ be the $j$-fold iterated logarithm; that is, $\log_2=\log\log$, $\log_3=\log\log\log$ and so on.
\begin{pro}\label{ShortApproxL}
Let $q$ be large and $0<\delta<1/2$ be fixed. Let $\chi$ be a non-principal character modulo $q$, and  $y$ be a real number in the range $\exp\left((\log_2 q)^3\right) \leq y \leq q$. Assume that $L(s,\chi)$ has no zeros inside the rectangle $\{s:1-\delta <\textup{Re}(s)\leq 1 \text{ and } |\textup{Im}(s)|\leq 2(\log q)^{2/\delta}\}$.
Then, for any complex number $z$ with $|z|\leq  (\log y)/(\log_2 q)^{2}$ we have
$$L(1,\chi)^z=\sum_{n=1}^{\infty}\frac{d_z(n)\chi(n)}{n}e^{-n/y}+O_{\delta}\left(\exp\left(-\frac{\log y}{2\log_2 q}\right)\right).$$
\end{pro}
To prove this result, we need the following lemma from \cite{DaLa}.
\begin{lem}[Lemma 3.1 of \cite{DaLa}]\label{BoundLZeroFree} Let $q$ be large and $\chi$ be a non-principal character modulo $q$. Put $\eta=1/\log_2q$, and let $0<\delta<1/2$ be fixed.  Assume that $L(z,\chi)$ has no zeros in the rectangle $\{z: 1-\delta < \textup{Re}(z)\leq 1\text{ and } |\textup{Im}(z)|\leq 2(\log q)^{2/\delta}\}.$ Then for any $s=\sigma+it$ with $1-\eta\leq \sigma\leq 1$ and $|t|\leq \log^4q$ we have
$$ |\log L(s,\chi)|\leq \log_3 q+O_{\delta}(1).$$
\end{lem}

\begin{proof}[Proof of Proposition \ref{ShortApproxL}]Since $\frac{1}{2\pi i}\int_{2-i\infty}^{2+i\infty}y^s\Gamma(s)ds= e^{-1/y}$ then
$$\frac{1}{2\pi i}\int_{2-i\infty}^{2+i\infty}L(1+s,\chi)^z\Gamma(s)y^sds= \sum_{n=1}^{\infty}\frac{d_z(n)\chi(n)}{n}e^{-n/y}.$$
we shift the contour to $\mathcal{C}$, where $\mathcal{C}$ is the path which joins
$$ -i\infty, -i(\log q)^4, -\eta-i(\log q)^4, -\eta+i(\log q)^4, -i(\log q)^4, +i\infty,$$
where $\eta=1/\log_2 q$.
By our assumption, we encounter only a simple pole at $s=0$ which leaves the residue $L(1,\chi)^z$. Also, since $\chi$ is a non-exceptional character, we can use the following standard bound (see for example Lemma 2.2 of \cite{La1})
\begin{equation}\label{StandardBound}
\log L(1+it, \chi)\ll \log_2 \big(q(|t|+2)\big).
\end{equation}
Using \eqref{StandardBound} together  with Stirling's formula we obtain
$$
\frac{1}{2\pi i}\left(\int_{-i\infty}^{-i(\log q)^4}+\int_{i(\log q)^4}^{i\infty}\right)L(1+s,\chi)^z\Gamma(s) y^s ds\ll \int_{(\log q)^4}^{\infty}e^{O(|z|\log_2 qt)}e^{-\frac{\pi}{3}t}dt\ll \frac{1}{q}.
$$
Finally, using that $\Gamma(s)$ has a simple pole at $s=0$ together with Stirling's formula and Lemma \ref{BoundLZeroFree}, we deduce that
\begin{align*}
&\frac{1}{2\pi i}\left(\int_{-i(\log q)^4}^{-\eta-i(\log q)^4}+\int_{-\eta-i(\log q)^4}^{-\eta+i(\log q)^4}+\int_{-\eta+i(\log q)^4}^{i(\log q)^4}\right)L(1+s,\chi)^z\Gamma(s)y^s ds\\
&\ll \exp\left(-\frac{\pi}{3}(\log q)^4+O(|z|\log_3 q)\right)+ \frac{y^{-\eta}}{\eta}\exp\big(|z|\log_3q+O_{\delta}(|z|)
\big)(\log q)^4\\
&\ll_{\delta} \exp\left(-\frac{\log y}{2\log_2 q}\right).
\end{align*}

\end{proof}

\section{Complex moments of $L(1,\chi_d)$ over $d\in \mc{P}(x)$: Proof of Theorem \ref{ComplexMoments}}

Let $\widetilde{\mc{P}}(x)$ be the set of discriminants $d=-p$ such that $\sqrt{x}\leq p\leq x$ is prime, $p\equiv 3 \bmod 4$ and  $L(s,\chi_{-p})$ has no zeros in the rectangle 
$\{s: 9/10 <\text{Re}(s)\leq 1 \text{ and } |\text{Im}(s)|\leq 2(\log x)^{20}\}$. To bound $|\mc{P}(x)\setminus \widetilde{\mc{P}}(x)|$ we use the following zero-density result of Heath-Brown \cite{HB}, which states that for $1/2<\sigma<1$ and
 any $\varepsilon>0$ we have
$$\sums_{|d|\leq x} N(\sigma,T, \chi_d)\ll (xT)^{\varepsilon}x^{3(1-\sigma)/(2-\sigma)}T^{(3-2\sigma)/(2-\sigma)},$$
where $N(\sigma, T, \chi_d)$ is the number of zeros $\rho$ of $L(s,\chi_d)$ with $\re(\rho)\geq \sigma$ and $|\im(\rho)|\leq T$, and $\sums$ indicates that the sum is over fundamental discriminants.  Using this bound we obtain
\begin{equation}\label{ExceptionalSize}
|\mc{P}(x)\setminus \widetilde{\mc{P}}(x)|\ll x^{1/2}.
\end{equation}

By  \eqref{StandardBound}, it follows that $\log L(1,\chi_{-p})\ll \log_2 p$  if $\chi_{-p}$ is a non-exceptional character. Since there is at most one exceptional prime modulus between any two powers of $2$ (see Chapter 14 of \cite{Da}), it follows that there are at most $O(\log x)$ exceptional characters $\chi_{-p}$ with $p\leq x$. In this case, we shall use the trivial bound $L(1,\chi_{-p})\gg p^{-1/2},$ which follows from the class number formula \eqref{Class}. Therefore, using \eqref{ExceptionalSize} and noting that $\re(z)\geq -1$ we obtain
 \begin{equation}\label{GoodL}
\sum_{d\in \mc{P}(x)\setminus \widetilde{\mc{P}}(x)} L(1,\chi_d)^z
\ll x^{1/2} \log x + x^{1/2} \exp\big(O(|z|\log_2 x)\big)\ll x^{2/3}.
\end{equation}
In the remaining part of the proof we let $k=\lfloor |z|\rfloor+1$. If $d\in \widetilde{\mc{P}}(x)$, then we can use Proposition \ref{ShortApproxL} in order to approximate $L(1,\chi_d)^z$. This gives 
\begin{equation}\label{ApproxGood}
\sum_{d\in \widetilde{\mc{P}}(x)} L(1,\chi_d)^z= \sum_{d\in \widetilde{\mc{P}}(x)}
\sum_{n=1}^{\infty}\frac{d_z(n)\chi_{d}(n)}{n}e^{-n/y}+  O\left(x\exp\left(-\frac{\sqrt{\log x}}{5\log\log x}\right)\right),
\end{equation} 
where $y=\exp\left(\frac12\sqrt{\log x}\right).$ We now extend the main term of the last estimate, so as to include all elements of $\mc{P}(x)$. Using  \eqref{Divisor2} and \eqref{ExceptionalSize}, we deduce that 
$$ 
\sum_{d\in \mc{P}(x)\setminus \widetilde{\mc{P}}(x)} \sum_{n=1}^{\infty}\frac{d_z(n)\chi_{d}(n)}{n}e^{-n/y} \ll x^{1/2} \sum_{n=1}^{\infty}\frac{d_k(n)}{n}e^{-n/y}\ll x^{1/2}(\log 2y)^k\ll x^{2/3}.
$$
Combining this estimate with \eqref{GoodL} and \eqref{ApproxGood}  gives
\begin{align*}
\sum_{d\in \mc{P}(x)} L(1,\chi_d)^z
&= \sum_{n=1}^{\infty} \frac{d_z(n)}{n} e^{-n/y}\sum_{d\in \mc{P}(x)} \chi_d(n) + O\left(x\exp\left(-\frac{\sqrt{\log x}}{5\log\log x}\right)\right)\\
&= \sum_{n\leq y(\log y)^2} \frac{d_z(n)}{n} e^{-n/y}\sum_{d\in \mc{P}(x)} \chi_d(n) + O\left(x\exp\left(-\frac{\sqrt{\log x}}{5\log\log x}\right)\right),
\end{align*}
since the contribution of the terms $n>y\log^2 y$ to the right hand side
is 
\begin{equation}\label{Tail}
\begin{aligned}
 \ll x \sum_{n>y \log^2 y}\frac{d_k(n)}{n}e^{-n/y} &\leq  x \exp\left(-\frac{(\log y)^2}{2}\right)\sum_{n=1}^{\infty}\frac{d_k(n)}{n}e^{-n/(2y)}\\
 &\ll  x^{7/8} (\log 4y)^k \ll x^{8/9},
\end{aligned}
\end{equation}
by \eqref{Divisor2}.
We are now able to use Lemma \ref{CharSum} to estimate the sum $\sum_{d\in \mc{P}(x)} \chi_d(n)$  since $n\leq y(\log y)^2\leq \exp\left(\sqrt{\log x}\right)$. Thus, Lemma \ref{CharSum} gives 
\begin{equation}\label{EstimateMo}
\begin{aligned} 
\sum_{d\in \mc{P}(x)} L(1,\chi_d)^z = & \ \frac{\textup{Li}(x)}{2}\sum_{m\leq \sqrt{y}\log y} \frac{d_z(m^2)}{m^2} e^{-m^2/y} \\
& -\textup{sgn}(q_1) \displaystyle{\frac{\textup{Li}(x^{\beta})}{2}} \sum_{m \leq \sqrt{y/|q_1|}\log y} \frac{d_z(|q_1| m^2)}{|q_1| m^2} e^{-|q_1|m^2/y} +O\left(\mc{E}_1(x)\right),\\
\end{aligned}
\end{equation}
where 
\begin{align*}
\mc{E}_1(x) & \ll x\exp\left(-c\sqrt{\log x}\right)\sum_{n=1}^{\infty} \frac{d_k(n)}{n} e^{-n/y} + x\exp\left(-\frac{\sqrt{\log x}}{5\log\log x}\right)\\ & \ll x\exp\left(-\frac{\sqrt{\log x}}{5\log\log x}\right) , 
\end{align*}
by \eqref{Divisor2}. Next, we use \eqref{Tail} to complete the two sums in the right hand side of \eqref{EstimateMo}. This yields
\begin{equation}\label{EstimateMo2}
\begin{aligned} 
\sum_{d\in \mc{P}(x)} L(1,\chi_d)^z= 
& \ \frac{\textup{Li}(x)}{2}\sum_{m=1}^{\infty} \frac{d_z(m^2)}{m^2} e^{-m^2/y} -\textup{sgn}(q_1) \displaystyle{\frac{\textup{Li}(x^{\beta})}{2}} \sum_{m=1}^{\infty} \frac{d_z(|q_1| m^2)}{|q_1| m^2} e^{-|q_1|m^2/y}\\
&  +O\left(x\exp\left(-\frac{\sqrt{\log x}}{5\log\log x}\right)\right).\\
\end{aligned}
\end{equation}
By \eqref{RandomExpectation1}, in order to complete the proof of Theorem \ref{ComplexMoments}, we need to replace the factors $e^{-m^2/y}$  and $e^{-|q_1|m^2/y}$ in the above sums by $1$,
and in so doing we introduce an error term of size at most
\begin{equation}\label{ErrorSmooth}
\mc{E}_2(x)=\textup{Li}(x)\left(\sum_{m=1}^{\infty} \frac{d_k(m^2)}{m^2}\left(1-e^{-m^2/y}\right)+ \sum_{m=1}^{\infty} \frac{d_k(|q_1|m^2)}{|q_1|m^2}\left(1-e^{-|q_1|m^2/y}\right)\right).
\end{equation}
We shall use the bound $1-e^{-t}\ll t^{\alpha}$ which is valid for all $t>0$ and  $0<\alpha\leq 1$.  Choosing $\alpha= 1/\log_2 x$, and using \eqref{BoundRandom} we deduce that 
$$
\mc{E}_2(x) \ll   y^{-\alpha} x\left(\sum_{m=1}^{\infty}\frac{d_k(m^2)}{m^{2-2\alpha}}+\sum_{m=1}^{\infty}\frac{d_k(|q_1|m^2)}{(|q_1|m^2)^{1-\alpha}}\right)\ll y^{-\alpha} x \cdot \sum_{m=1}^{\infty}\frac{d_k(m^2)}{m^{2-2\alpha}}, 
$$
 Finally, using the bound \eqref{Divisor3} and noting that $d_k(m^2)\leq d_k(m)^2$ for all integers $m\geq 1$, we obtain
$$  \sum_{m=1}^{\infty}\frac{d_k(m^2)}{m^{2-2\alpha}}\leq \sum_{m=1}^{\infty}\frac{d_k(m)^2}{m^{2-2\alpha}}\leq \exp\left((2+o(1)) k\log_2 k \right)\ll \exp\left(\frac{\sqrt{\log x}}{20\log_2x}\right).
$$
This implies that 
$ \mc{E}_2(x) \ll \exp\left(-\sqrt{\log x}/(5\log_2 x)\right).$
Combining this estimate with \eqref{EstimateMo2} completes the proof of Theorem \ref{ComplexMoments}.

\section{Proof of Theorem \ref{AverageClass}}

To prove Theorem \ref{AverageClass} we shall follow the argument in \cite{La2}, which is a refinement of the work of Soundararajan \cite{So}. 
\begin{lem}\label{SmoothPerron}
Let $\lambda, c >0$ be real numbers and $N\geq 0$ be an integer. For $y>0$ we define
$$
I_{c, \lambda, N}(y):=\frac{1}{2\pi i}\int_{c-i\infty}^{c+i\infty} y^s \left(\frac{e^{\lambda s}-1}{\lambda s}\right)^N\frac{ds}{s},
$$
Then we have 
$$
I_{c, \lambda, N}(y)
\begin{cases}=1 & \text{ if } y> 1, \\ \in [0,1] & \text{ if } e^{-\lambda N } \leq y \leq 1,\\
 =0 & \text{ if } 0<y< e^{-\lambda N }. \end{cases}
$$
\end{lem}
\begin{proof}
The result follows from Perron's formula together with the following identity 
$$\frac{1}{2\pi i}\int_{c-i\infty}^{c+i\infty} y^s \left(\frac{e^{\lambda s}-1}{\lambda s}\right)^N\frac{ds}{s}
= \frac{1}{\lambda^N}\int_{0}^{\lambda}\cdots \int_0^{\lambda} \frac{1}{2\pi i} \int_{c-i\infty}^{c+i\infty}
\left(ye^{t_1+ \cdots+ t_N}\right)^s\frac{ds}{s} dt_1\cdots dt_N
$$ for $N\geq 1$.
\end{proof}

\begin{proof}[Proof of Theorem \ref{AverageClass}]
In order to obtain an asymptotic formula for $\sum_{\substack{h\leq H\\ h  \text{ odd}}}\F(h)$,
we first show that we can restrict our attention to discriminants $d\in \mathcal{P}(X)$ with  $X:=H^2 (\log H)^5$.  Indeed, if $-d \geq X$ and $h(d)\leq H$ then by the class number formula \eqref{Class} we must have $L(1,\chi_d)\ll 1/(\log H)^{5/2}$. However, it follows from Tatuzawa's refinement of Siegel's Theorem \cite{Ta} that for large $|d|$, we have $L(1,\chi_d)\geq 1/(\log |d|)^2$ with at most one exception. Thus
we obtain
$$\sum_{\substack{h\leq H\\ h  \text{ odd}}}\F(h)= \sum_{\substack{d \in \mathcal{P}(X)\\ h(d) \leq H}} 1
	+ O(1).
$$
Let $c=1/\log H$, $N$ be a positive integer, and $0<\lambda \leq 1/N$ be a real number to be chosen later. Then it follows from Lemma \ref{SmoothPerron} that 
\begin{equation}\label{Bounds}
\sum_{\substack{h\leq H\\ h  \text{ odd}}}\F(h)
\leq \frac{1}{2\pi i}\int_{c-i\infty}^{c+i\infty}\sum_{d\in \mc{P}(X)}\frac{H^s}{h(d)^s}\left(\frac{e^{\lambda s}-1}{\lambda s}\right)^N\frac{ds}{s} +O(1)\leq \sum_{\substack{h\leq e^{\lambda N}H\\ h  \text{ odd}}}\F(h).
\end{equation}
Let $T:= \sqrt{\log X}/(\log_2 X)^2.$ 
Since  $|e^{\lambda s}-1|\leq e^{\lambda c}+1\leq 3$ if $H$ is large enough,  and $h(d)\geq 1$, it follows that the contribution of the region $|s|>T$ to the integral in \eqref{Bounds} is
\begin{equation}\label{ErrorCut0}
\ll X\left(\frac{3}{\lambda}\right)^N \int_{\substack{|s|> T\\ \re(s)=c}} \frac{|ds|}{|s|^{N+1}}\ll \frac{X}{N}\left(\frac{3}{\lambda T}\right)^N.
\end{equation}

By partial summation and \eqref{Class}, it follows from Theorem \ref{ComplexMoments} that for all complex numbers $s$ such that $\re(s)=c$ and $|s|\leq T$ we have
\begin{equation}\label{Moments}
\begin{aligned}
\sum_{d\in \mc{P}(X)} h(d)^{-s}= & \ \frac{\pi^s}{2}\cdot \ex\left(L(1,\X)^{-s}\right)\int_2^X x^{-s/2}d \textup{Li}(x) \\
& - \frac{\pi^s}{2} \textup{sgn}(q_1) \cdot \ex\big(\X(|q_1|)\cdot L(1,\X)^{-s}\big)\int_2^X x^{-s/2}d \textup{Li}(x^{\beta})\\
 & +O\left(X\exp\left(-\frac{\sqrt{\log X}}{6\log_2 X}\right)\right).
 \end{aligned}
\end{equation}
Combining \eqref{ErrorCut0} and \eqref{Moments} shows that the integral in \eqref{Bounds} equals
\begin{equation}\label{MTerm}
\begin{aligned}
&\frac{1}{2\pi i}\int_{\substack{|s|\leq T\\ \re(s)=c}} \frac{1}{2}\cdot \ex\left(\left(\frac{\pi H}{L(1,\X)}\right)^{s}\right)\left(\int_2^X x^{-s/2}d \textup{Li}(x)\right) \left(\frac{e^{\lambda s}-1}{\lambda s}\right)^N\frac{ds}{s} \\
 & 
 -\frac{1}{2\pi i}\int_{\substack{|s|\leq T\\ \re(s)=c}} \frac{\textup{sgn}(q_1)}{2}\cdot \ex\left(\X(|q_1|)\cdot \left(\frac{\pi H}{L(1,\X)}\right)^{s}\right)\left(\int_2^X x^{-s/2}d\textup{Li}(x^{\beta})\right)
 \left(\frac{e^{\lambda s}-1}{\lambda s}\right)^N\frac{ds}{s} + \mathcal{E}_3,
 \end{aligned}
 \end{equation}
 where 
 $$\mathcal{E}_3\ll \frac{X}{N}\left(\frac{3}{\lambda T}\right)^N+ \frac{3^N T }{c} X\exp\left(-\frac{\sqrt{\log X}}{6\log_2 X}\right),$$
 since $|(e^{\lambda s}-1)/\lambda s|\leq 3$ if $H$ is large enough.
 Choosing $\lambda = 10/T$ and $N=[10\log_2 H]$,  implies that\begin{equation}\label{ErrorCut}
\mathcal{E}_3\ll \frac{H^2}{(\log H)^{10}}.
\end{equation}

We now extend the integrals in \eqref{MTerm} to $\int_{c-i\infty}^{c+i\infty}$, and in so doing  we introduce an error term $\mathcal{E}_4$, where similarly to \eqref{ErrorCut0} we have
$$ \mathcal{E}_4\ll \ex\left(L(1,\X)^{-c}\right)\frac{\textup{Li}(X)}{N}\left(\frac{3}{\lambda T}\right)^N\ll\frac{H^2}{(\log H)^{10}}.
$$
Therefore, we deduce that the integral in \eqref{Bounds} equals
\begin{equation}\label{ExpectationComputation}
\begin{aligned}
&\frac{1}{2}\cdot \ex\left(\int_2^X I_{c, \lambda, N} \left(\frac{\pi H }{\sqrt{x}}L(1, \X)^{-1}\right)d \textup{Li}(x)\right)\\
&-\frac{\textup{sgn}(q_1)}{2}\ex\left(\X(|q_1|)\int_2^X I_{c, \lambda, N} \left(\frac{\pi H }{\sqrt{x}}L(1, \X)^{-1}\right)d \textup{Li}(x^{\beta})\right)+ O\left(\frac{H^2}{(\log H)^{10}}\right).
\end{aligned}
\end{equation}
 Now, it follows from Lemma \ref{SmoothPerron} that for any $1\leq x\leq X$ we have
$$ I_{c, \lambda, N} \left(\frac{\pi H }{\sqrt{x}}L(1, \X)^{-1}\right)= \begin{cases} 1 & \text{ if } \sqrt{x}L(1, \X)\leq \pi H , \\ \in [0,1] & \text{ if } \pi H< \sqrt{x}L(1, \X)\leq e^{\lambda N }\pi H,\\
0 & \text{ if } \sqrt{x} L(1, \X)> \pi H e^{\lambda N }. \end{cases}$$
Thus we obtain
\begin{equation}\label{Expectation}
\begin{aligned}
 &\ex\left(\int_2^X I_{c, \lambda, N} \left(\frac{\pi H }{\sqrt{x}}L(1, \X)^{-1}\right)d \textup{Li}(x) \right) \\
&=\ex\left(\textup{Li}\left(\min\left( \frac{\pi^2 H^2}{L(1,\X)^2}, X\right)\right) +O\left(\int_{\pi^2H^2/L(1, \X)^2}^{e^{2\lambda N}\pi^2H^2/L(1, \X)^2} d \textup{Li}(x)\right)\right)\\
&= \ex\left(\textup{Li}\left(\min\left( \frac{\pi^2 H^2}{L(1,\X)^2}, X\right)\right)\right)+ O\left(\frac{H^2(\log_2 H)^3}{(\log H)^{3/2}}\right),
\end{aligned}
\end{equation}
by \eqref{LargeDeviation} together with the fact that $e^{2\lambda N}-1 \ll (\log_2 H)^3/\sqrt{\log H}$.
Furthermore,  it follows from \eqref{LargeDeviation2} that 
$$ \mathbb{P}\left(L(1, \X)\leq \frac{\pi H}{\sqrt{X}}\right)\ll e^{-X}.$$
Therefore, we get
\begin{align*}
 \ex\left(\textup{Li}\left(\min\left( \frac{\pi^2 H^2}{L(1,\X)^2}, X\right)\right)\right)
 &= \ex\left(\textup{Li}\left(\frac{\pi^2 H^2}{L(1,\X)^2}\right)\right) +O\left(1\right)\\
 &= \frac{\pi^2 H^2}{2\log H} \cdot \ex\left(L(1,\X)^{-2}\right)+ O\left(\frac{H^2}{(\log H)^2}\right).
\end{align*}
Inserting this estimate in \eqref{Expectation} gives
\begin{equation}\label{MainIntegralPerron}
\ex\left(\int_2^X I_{c, \lambda, N} \left(\frac{\pi H }{\sqrt{x}}L(1, \X)^{-1}\right)d \textup{Li}(x)\right)=  \frac{\pi^2 H^2}{2\log H} \cdot \ex\left(L(1,\X)^{-2}\right)+O\left(\frac{H^2(\log_2 H)^3}{(\log H)^{3/2}}\right).
\end{equation}
Using the same argument, we also derive
\begin{equation}\label{SiegelExtra}
\begin{aligned} 
&\ex\left(\X(|q_1|)\int_2^X I_{c, \lambda, N} \left(\frac{\pi H }{\sqrt{x}}L(1, \X)^{-1}\right)d \textup{Li}(x^{\beta})\right)\\
&= \frac{(\pi H)^{2\beta}}{2\beta \log H}\cdot \ex\left(\X(|q_1|)L(1,\X)^{-2\beta}\right)+O\left(\frac{H^2(\log_2 H)^3}{(\log H)^{3/2}}\right).
\end{aligned}
\end{equation}
A simple computation shows that 
$$ 
\ex\left(L(1,\X)^{-2}\right)= \prod_{p} \left(1-\frac{1}{p^4}\right)\left(1-\frac{1}{p^2}\right)^{-1}=\frac{\zeta(2)}{\zeta(4)}=\frac{15}{\pi^2}.
$$
On the other hand, since $d_{2\beta}(n)\leq d(n)$ by \eqref{Divisor1} (where $d(n)=d_2(n)$ is the number of divisors of $n$), and $|q_1|$ is square-free, then it follows from  \eqref{RandomExpectation1} and \eqref{RandomExpectation2} that
$$
\ex\left(\X(|q_1|)L(1,\X)^{-2\beta}\right) \leq \ex\left(\X(|q_1|)L(1,\X)^{-2}\right)= \ex\left(L(1,\X)^{-2}\right) \prod_{p\mid q_1} \frac{4 p}{2p^2+2}\ll \frac{d(|q_1|)}{|q_1|}.
$$
Moreover, since $d(|q_1|)=|q_1|^{o(1)}$ and $|q_1|\gg (\log X)/(\log_2 X)^4\gg (\log H)/(\log_2 H)^4$ (see Chapter 20 of \cite{Da}) then
$$ \ex\left(\X(|q_1|)L(1,\X)^{-2\beta}\right)\ll_{\varepsilon} \frac{1}{(\log H)^{1-\varepsilon}}.$$
Inserting this estimate in \eqref{SiegelExtra}, and using \eqref{Bounds},  \eqref{ExpectationComputation}, and \eqref{MainIntegralPerron} we deduce that
$$ \sum_{\substack{h\leq H\\ h  \text{ odd}}}\F(h)\leq \frac{15 H^2}{4\log H}+ O\left(\frac{H^2(\log\log H)^3}{(\log H)^{3/2}}\right)\leq \sum_{\substack{h\leq e^{\lambda N}H\\ h  \text{ odd}}}\F(h).$$
Using the same inequality with $e^{\lambda N}H$ instead of $H$, and noting that $e^{2\lambda N}-1 \ll (\log_2 H)^3/\sqrt{\log H}$ completes the proof.

 \end{proof}

\end{document}